\documentclass{siamltex}

\usepackage{amssymb}
\usepackage{amsmath}
\usepackage{graphicx, enumerate}
\usepackage{psfrag}
\usepackage{array}
\usepackage[square, numbers]{natbib}
\usepackage[unicode]{hyperref}
\usepackage{hypernat}
\usepackage{subfigure}
\usepackage{algorithm, algorithmic}

\newcommand{\R}{\mathbb{R}}

\newcommand{\N}{\mathbb{N}}

\newcommand{\Rb}{\overline{\R}}

\newcommand{\Ls}[1]{\textup{L}^{#1}}
\renewcommand{\L}[1]{\textup{L}^{#1}(\Omega)}

\newcommand{\BV}{\textup{BV}(\Omega)}
\newcommand{\TV}{\textup{BV}}

\newcommand{\D}{\text{D}}

\DeclareMathOperator*{\argmin}{\textnormal{argmin}}

\DeclareMathOperator*{\wlim}{{\it w}\textnormal{-lim}}

\newcommand{\eps}{\varepsilon}

\renewcommand{\ker}{\textnormal{ker}}
\newcommand{\ran}{\textnormal{ran}}

\newcommand{\set}[1]{\left\{ #1 \right\}}
\newcommand{\abs}[1]{\left| #1 \right|}
\newcommand{\norm}[1]{\left\| #1 \right\|}
\newcommand{\inner}[2]{\left\langle #1, #2 \right\rangle}

\newcommand{\bigo}{\mathcal{O}}
\newcommand{\ra}{\rightarrow}

\newcommand{\Gr}{\textnormal{Gr}}

\newcommand{\Reg}{\mathcal{R}}

\newtheorem{ass}{Assumption}

\newtheorem{remark}[theorem]{Remark \upshape}
\newtheorem{alg}{Algorithm}
\newtheorem{example}{Example}
\title{Morozov's principle for the augmented Lagrangian method
  applied to linear inverse problems}

\author{Klaus Frick\footnotemark[2] \and Dirk A. Lorenz\footnotemark[3]
  \and Elena Resmerita\footnotemark[4]}

\date{\today}
\begin{document}
\maketitle
\renewcommand{\thefootnote}{\fnsymbol{footnote}}
\footnotetext[2]{Institute for Mathematical
    Stochastics, University of G\"ottingen, Goldschmidtstra{\ss}e 7
    D-37077 G\"ottingen, Germany,
    \url{frick@math.uni-goettingen.de}}
\footnotetext[3]{Institute for Analysis and Algebra, TU Braunschweig,
   D-38092 Braunschweig, Germany,
   \url{d.lorenz@tu-braunschweig.de}}
\footnotetext[4]{Industrial Mathematics Institute, 
 Johannes Kepler University Linz,
 Altenbergerstraße 69,
 A-4040 Linz, Austria,
 \url{elena.resmerita@jku.at}
}
\renewcommand{\thefootnote}{\arabic{footnote}}
\pagestyle{myheadings}
\thispagestyle{plain}
\markboth{\uppercase{K.~Frick, D.A.~Lorenz, and
    E. Resmerita}}{\uppercase{Morozov's principle for the augmented
    Lagrangian method}}

\begin{abstract}
 The Augmented Lagrangian Method as an approach for regularizing inverse
 problems received much attention recently,
 e.g.~under the name Bregman iteration in imaging.
 This work shows convergence (rates) for this method when Morozov's discrepancy
 principle is chosen as a stopping rule. Moreover, error estimates for the
 involved sequence of subgradients are pointed out. 
 
 The paper studies implications of these results for particular examples motivated by applications in imaging. These include the
 total variation regularization as well as $\ell^q$ penalties with
 $q\in[1,2]$. It is shown that Morozov's principle implies convergence (rates)
 for the iterates with  respect to the metric of strict convergence and the
 $\ell^q$-norm, respectively.
\end{abstract}
%\subjclass[2000]{Primary primary class 65J22; Secondary secondary 46N10, 49M05}
%46Nxx Miscellaneous applications of functional analysis
%46N10 Applications in optimization, convex analysis, mathematical programming, economics
%
%49Mxx Methods of successive approximations
%49M05 Methods based on necessary conditions 
%
%65Jxx Numerical analysis in abstract spaces
%65J22 Inverse problems
%\keywords{Augmented Lagrangian Method, Bregman iteration, Morozov's
%  discrepancy principle, regularization}

\section{Introduction}\label{intro}  

A classical problem in optimization is the solution of
\begin{equation}\label{intro:constrained}
  J(u) \to \min \quad \text{ subject to } \quad Ku=g\,,  
\end{equation}
where $J:H_1\ra \R\cup\{\infty\}$ is a convex functional and $K: H_1 \ra H_2$ is
a linear and bounded operator between Hilbert spaces $H_1$ and $H_2$.
Solutions of problem \eqref{intro:constrained} are called \emph{$J$-minimizing
solutions} of the equation $Ku=g$.
% \begin{equation}\label{intro:opeqn}
%   Ku = g.
% \end{equation}

Of particular interest are ill-posed equations, that is, when the solution of
$Ku=g$ does not depend continuously on the data $g$ (as it is
e.g.~the case if $K$ has non-closed range). This becomes distinctly delicate if the data $g$ is
not available precisely but only noise-affected observations $g^\delta$ for which we assume
that we have the additional information $\norm{g^\delta-g}\leq \delta$.
% \begin{equation}\label{intro:delta}
%   \norm{g^\delta-g} \leq \delta.
% \end{equation}

It is a natural question to ask: ``When does a solution algorithm for the
optimization problem \eqref{intro:constrained} applied to perturbed data $g^\delta$ instead
of $g$, constitute a \emph{regularization method} for the ill-posed equation
$Ku=g$?'' In \cite{FriSch10} an affirmative answer was given for the
\emph{Augmented Lagrangian Method} (ALM), which in the context of
regularization is also known as the \emph{Bregman iteration} (see
\cite{OshBurGolXuYin05}). The ALM was introduced simultaneously by Hestenes
\cite{Hes69} and Powell \cite{Pow69} as an iterative solution method for \eqref{intro:constrained} and reads as follows:

\begin{alg}[the ALM]\label{intro:ala} 
  Let $p_0^\delta \in H_2$ and choose a sequence $\set{\tau_n}_{n\in\N}$ of
  positive parameters. For $n= 1,2,\ldots$ compute
  \begin{subequations}
    \begin{align}
      u_n^\delta & \in \argmin_{u \in H_1} \left( \frac{\tau_n}{2} \norm{Ku -
          g^\delta}^2 + J(u) - \inner{p_{n-1}^\delta}{Ku - g^\delta} \right)\quad\text{
        and }\label{intro:almobjective} \\
      p_n^\delta & = p_{n-1}^\delta + \tau_n (g^\delta - K
      u_n^\delta).\label{intro:almupdate}
    \end{align}
  \end{subequations}
\end{alg}
The name \emph{Augmented Lagrangian} stems from the fact that the functional
\begin{equation*}
  \mathcal{L}(u,p) = J(u) - \inner{p}{Ku-g^\delta}
\end{equation*} 
is the Lagrangian for~\eqref{intro:constrained} and the additional term
$\frac{\tau_n}{2} \norm{Ku - g^\delta}^2$ is an \emph{augmentation} of
$\mathcal{L}$ that fosters the fulfillment of the constraint. Hence, in the
limit, the augmentation term is supposed to vanish and the
variables $p_n^\delta$ shall tend to a Lagrange multiplier for the
problem~\eqref{intro:constrained}. 

It is well known that the Karush-Kuhn-Tucker conditions are necessary and
sufficient regularity conditions for the solutions of
\eqref{intro:constrained}, which guarantee existence of a saddle point of
$\mathcal{L}$. Thus, if there exists $u^\dagger\in H_1$ and
$p^\dagger\in H_2$ such that
\begin{equation*}
  Ku^\dagger = g\quad\text{ and }\quad K^*p^\dagger \in \partial J(u^\dagger)
\end{equation*}
then,
$\mathcal{L}(u^\dagger,p)\leq\mathcal{L}(u^\dagger,p^\dagger)\leq\mathcal{L}(u,p^\dagger)
$. It was pointed out in ~\cite{BurOsh04} that this coincides with the
standard source conditions in regularization theory.

As in \cite{FriSch10}, we will consider the ALM as a regularization method, that is, for stably computing
approximations of solutions of \eqref{intro:constrained} from perturbed
data $g^\delta$. With $\Reg_n: H_2 \ra H_1$ and $\Reg^*_n:H_2\ra H_2$
we denote the operators defined by
\begin{equation*}
  \Reg_n(g^\delta) := u_n^\delta\ \text{ and }\ \Reg^*_n(g^\delta) = p_n^\delta,\ \text{respectively}.
\end{equation*}
The paper \cite{FriSch10} came up with a characterization of parameter choice
rules $\Gamma:(0,\infty)\times H_2 \ra \N$ such that for each solution
$u^\dagger$ of \eqref{intro:constrained}
\begin{equation*}
  \Reg_{\Gamma(\delta_k, g_k)}(g_k) \ra u^\dagger\quad\text{ as }\quad
  \norm{g-g_k} =: \delta_k \ra 0
\end{equation*}
in an appropriate sense. Under a standard source condition, it showed also convergence rates for a class of stopping rules $\Gamma(\delta,y^\delta)$  for which $\Gamma(\delta,y^\delta)\ra \infty$, as $\delta\ra 0$.
We pursue further that study and mainly show that Morozov's discrepancy principle does belong to the above mentioned class. Moreover, we investigate the degenerate case of the discrepancy principle, that is when $\{\Gamma(\delta,g^\delta)\}$ has finite accumulation points. Note that the complex challenge of choosing a right regularization parameter when dealing with  stabilization methods for improperly posed problems is frequently approached via Morozov's rule due to its natural heuristic  motivation. Namely, this rule selects a parameter by comparing the residual $\norm{K u_n^\delta - g^\delta}$ with the presumably known noise level $\delta$ - see, e.g. \cite[Ch. 4]{EngHanNeu96}.
% Interestingly, the ALM iterates in this case not only cluster to solutions of the equation as $\delta\ra 0$, but regularity of these solutions is granted as well, in the sense that they automatically satisfy the source condition and thus, lead to convergence rates.  

In \cite{FriSch10}, the implications of general convergence analysis for the ALM
were emphasized for the case of quadratic functionals $J$ (cf. Example
\ref{defs:quad}). In particular, the authors pointed out that in this case the
ALM is equivalent to the \emph{Tikhonov-Morozov} method (cf. \cite{Gro07}). Here, we will study in more
detail two choices for $J$ that are especially appealing for inverse
problems occurring in \emph{imaging}: 
\begin{enumerate}[i)]
 \item \emph{Total-variation regularization} (cf.
  \cite{RudOshFat92,BurOsh04,BurResHe07}). Let $H_1 = \L{2}$
  for a bounded
  domain $\Omega \subset \R^2$ and consider the function
  \begin{equation}\label{intro:tv}
    J(u) = \begin{cases}
\abs{\D u}(\Omega) & \text{ if } u \in \BV \\
+\infty & \text{ else}.
\end{cases}
  \end{equation}
  Here, $\abs{\D u}(\Omega)$ denotes the total-variation of the (measure-valued)
  distributional derivative of $u$.
  \item \emph{Sparse regularization} (cf. \cite{DauDefDeM04,Lor08,GraHalSch08}).
  Let $H_1=\ell^2$ and
  \begin{equation}\label{intro:sparse}
    J(u) =
    \begin{cases}
      \sum_{k\in\N} \abs{u_k}^q & \text{if }\ u\in\ell^q\\
      +\infty & \text{otherwise}
    \end{cases}
  \end{equation}
  with $1\leq q\leq 2$.
\end{enumerate}
This work is organized as follows. Section 2 presents the main notions and
notation, while Section 3 recalls several results in \cite{FriSch10} and proposes
some extensions of them. For instance, upper bounds for the Bregman distance
between the subgradients of the objective functional $J$ in
\eqref{intro:constrained} corresponding to the iterates and the solution,
respectively, are obtained. Section 4 shows that the ALM together with Morozov's
discrepancy principle lead to stable approximations for the operator equation
both in the nondegenerate and degenerate cases. The results are applied for the
total variation setting in Section 5, by underlying  strict convergence (rates)
for the primal variables. Section 6  summarizes the knowledge on the ALM for the
sparsity regularization setting, i.e.~convergence rates for the primal variables
with respect to the $\ell^q$-norm and for the subgradients of these variables
with respect to Bregman distances ($1\leq q\leq 2$) and  dual norms ($1<q<2$).

%%%%%%%%%%%%%%%%%%%%%%%%%%%%%%%%%%%%%%%%%%%%%%%%%%%%%%%%%%%%%%%%%%%%%%%%%%%%%%%
\section{Basic Definitions and some Notation}\label{defs}

\subsection{Basic Assumptions} Throughout this paper we will assume that $H_1$
and $H_2$ are separable Hilbert spaces with inner products
$\inner{\cdot}{\cdot}$ and norms $\norm{\cdot}$ (not further specified since the meaning is always
clear from the context). We will frequently make use of \emph{Young's
inequality}, which states that for all $u,v \in H_1$ and $\gamma > 0$ one has
that
\begin{equation*}
  \abs{\inner{u}{v}}\leq \frac{1}{2\gamma}\norm{u}^2 +
  \frac{\gamma}{2}\norm{v}^2.
\end{equation*}

We assume further that $K:H_1 \ra H_2$ is a linear
and bounded operator and that $J:H_1\ra \Rb = \R\cup\set{\infty}$ is convex,
lower semi-continuous (l.s.c.) and proper, that is, the domain
\begin{equation*}
  D(J) = \set{u \in H_1~:~ J(u) < \infty}
\end{equation*}
is non-empty. In order to guarantee that $J$-minimizing solutions of
$Ku=g$ exist and that Algorithm \ref{intro:ala} is well defined,
we need to impose additional restrictions (cf. \cite[Lem. 3.1]{FriSch10}):

\begin{ass}\label{defs:compactnessass}
The sub-level sets of the functional
\begin{equation*}
  u \mapsto \norm{Ku}^2 + J(u)
\end{equation*}
are sequentially pre-compact with respect to the weak topology on $H_1$. That
is, for every $c\in\R$, every sequence $\set{u_n}_{n\in\N}$ contained in the
sub-level set
\begin{equation*}
  \Lambda(c) = \set{u \in H_1~:~ \norm{Ku}^2 + J(u) \leq c}
\end{equation*}
has a weakly convergent subsequence in $H_1$.
\end{ass}

Moreover, we will assume that $\set{\tau_n}_{n\in\N}$ in Algorithm \ref{intro:ala}
is a fixed sequence of positive regularization parameters which can be considered
as step-sizes for the iterations.  We will make use of the quantity
\begin{equation*}\label{summary:tn_taun}
  t_n := \sum_{k=1}^n \tau_k.
\end{equation*}
The case of constant parameter $\tau_n = \tau$ is known as stationary augmented
Lagrangian method and leads to $t_n = n\tau$. We will only require that
\begin{equation}\label{infty}
  \lim_{n\ra\infty} t_n =+ \infty\quad\text{ and }\quad \sup_{k\in\N}\tau_k
  =:\bar\tau < \infty,
\end{equation}
i.e., the $\tau_n$'s do not decay too quickly and stay bounded.

Finally, we will assume that $g\in H_2$ is an \emph{attainable} element, that
is, there exists a $u\in D(J)$ such that $Ku=g$. By
$g^\delta\in H_2$ we always denote a perturbed version of $g$ satisfying
$\norm{g^\delta-g}\leq \delta$. For $k\in\N$, we will abbreviate $g_k:= g^{\delta_k}$ with
$\delta_k \ra 0$ as $k\ra \infty$.

\subsection{Convex Analysis} In the course of this paper we will frequently
 use some tools from convex analysis. A standard reference in this respect is \cite{EkeTem76}.

The {\it subdifferential} (or generalized derivative)  $\partial J(u)$ of $J$ at
$u$ is the set of all elements $\xi \in H_1$ satisfying
\begin{equation*}
  J(v) - J(u) - \inner{\xi}{v-u} \geq 0.
\end{equation*}
The \emph{domain} $D(\partial J)$ of the subgradient consists of all $u\in H_1$
for which $\partial J (u) \neq \emptyset$. Finally, we define the \emph{graph} of
$\partial J$ as
\begin{equation*}
  \Gr(\partial J) := \set{(u,\xi) \in H_1\times H_1 ~:~ \xi\in \partial J (u)}.
\end{equation*}
According to \cite[Chap.~I Cor.~5.1]{EkeTem76}, the set $\Gr(\partial J)$ is
sequentially closed with respect to the weak-strong topology on $H_1\times H_1$. That is, if
the sequence $\set{(u_n,v_n)}_{n \in \N}$ of elements in $\Gr(\partial J)$
satisfies that $u_n$ converges weakly to $u$ and $v_n$ converges strongly to $v$,
then $(u,v) \in \Gr(\partial J)$.

The functional $J^*: H_1 \to \Rb$ denotes the
{\it Legendre-Fenchel transform} (or the dual functional) of $J$, which is
defined by
\begin{equation*}\label{defs:legfen}
  J^*(v) := \sup_{u \in H_1} (\inner{v}{u} - J(u)).
\end{equation*}
Since $J^*$ is the pointwise supremum of
affine functions it is convex, l.s.c.~ and
proper~\cite[Chap.~I, Prop.~3.1]{EkeTem76}.
Moreover, one has \cite[Chap.~I, Cor.~5.2.]{EkeTem76}
\begin{equation*}\label{defs:subgrfenchel}
  v\in\partial J(u) \Leftrightarrow u \in \partial J^*(v).
\end{equation*}
Furthermore, it follows from the definition of the subgradient that
\begin{equation*}
  u \in \partial J^*(K^*p)\Rightarrow Ku\in \partial(J^*\circ K^*)(p).
\end{equation*}

For $u \in D(\partial J)$ and $v\in D(J)$, the {\it Bregman
distance} of $J$ between $u$ and $v$ with respect to $\xi \in \partial J(u)$
is defined by
\begin{equation*}
  D_J^{\xi}(v,u) = J(v) - J(u) - \inner{\xi}{v-u}.
\end{equation*}
We will skip the superscript $\xi$, if the choice of the subgradient is
obvious. If additionally $v\in D(\partial J)$ and $\eta
\in \partial J(v)$, we further define the \emph{symmetric} Bregman distance, by
\begin{equation*}
  D_J^{\text{sym}}(v,u) = D_J(v,u) + D_J(u,v) = \inner{\eta - \xi}{v - u}.
\end{equation*}
Note that the convexity of $J$ implies that $D_J$ and $D_J^{\text{sym}}$
are always non-negative.

\begin{example}\label{defs:quad}
  Let $H$ be a Hilbert space and $L:D(L)\subset H_1\ra H$ be a linear
  and closed operator with dense domain $D(L)$. Then, the \emph{quadratic
  functional}
  \begin{equation*}
      J(u) = \begin{cases} \frac{1}{2}\norm{Lu}^2 & \text{ if } u\in D(L) \\
             +\infty & \text{ else.}
             \end{cases}
  \end{equation*}
  is convex, lower semi-continuous and proper. Moreover, for $u\in D(\partial J)
  = D(L^*L)$ the subgradient $\partial J(u)$ coincides with the set
  $\set{L^*Lu}$ (cf. \cite[Lem. 2.4]{FriSch10}). This finally implies that 
  \begin{equation*}
      D_J^{\text{sym}}(v,u) = \norm{L(v-u)}^2.
  \end{equation*}
\end{example}

\subsection{Source Condition}\label{defs:sc}
It is well known, that
regularization methods for the reconstruction of a solution $u^\dagger$ of \eqref{intro:constrained} in
general converge arbitrarily slow, unless further regularity is imposed on
$u^\dagger$~\cite{EngHanNeu96}. In the general setup presented in this paper, this is usually done
in terms of the standard \emph{source condition}~\cite{BurOsh04},
that is, there exists an
element $p^\dagger \in H_2$ (the source element) such that
\begin{equation}\label{defs:sceqn}
  K^*p^\dagger \in \partial J(u^\dagger).
\end{equation}

%%%%%%%%%%%%%%%%%%%%%%%%%%%%%%%%%%%%%%%%%%%%%%%%%%%%%%%%%%%%%%%%%%%%%%%%%%%%%%%
\section{Summary and extensions of previous results}\label{summary}

In this section we summarize the results on regularization by means of the ALM
as presented in \cite{FriSch10}. We further derive an extended error estimate
that allows for convergence rates of the sequence $K^*p_n^\delta$ in the
Bregman-distance associated with the Fenchel conjugate $J^*$.

The dual characterization of the ALM by the
\emph{proximal point method} plays a central role in the convergence analysis in
\cite{FriSch10}. This observation dates back to the work of Rockafellar in
\cite{Roc74}. In the current context, defining $G:H_2 \times H_2 \ra \R$ by
\begin{equation}\label{summary:descent}
  G(p,g) = J^*(K^*p) - \inner{p}{g},
\end{equation}
it holds (cf. \cite[Prop. 4.2]{FriSch10})
\begin{equation}\label{summary:dual}
  p_n^\delta = \argmin_{p \in H_2}\left( \frac{1}{2} \norm{p -
  p_{n-1}^\delta}^2 + \tau_n G(p;g^\delta) \right).
\end{equation}
The basis of the results in \cite{FriSch10} is the following estimate on the
iterates in \eqref{summary:dual} which was established by G\"uler
in~\cite[Lem.~2.2]{Gue91}:

\begin{proposition}\label{summary:gueler}
For all $n\in\N$ and all $p\in H_2$ one has
\begin{equation}\label{summary:guelest}
  G(p_n^\delta,g^\delta) - G(p,g^\delta)
  \leq \frac{\norm{p-p_0^\delta}^2}{2 t_n} - \frac{\norm{p-p_n^\delta}^2}{2t_n} -
  \frac{t_n \norm{p_n^\delta -p_{n-1}^\delta}^2}{2\tau_n^2}.
\end{equation}
\end{proposition}

This result leads to the general convergence result \cite[Thm. 5.3]{FriSch10}:
\begin{theorem}\label{summary:conv}
  Assume that the stopping rule $\Gamma:(0,\infty)\times H_2
  \ra \N$ satisfies
  \begin{equation}\label{summary:abbrev}
    \lim_{k\ra\infty} \delta_k^2 t_{\Gamma(\delta_k, g_k)} = 0 \quad \text{ and
    }\quad \lim_{k\ra\infty}  t_{\Gamma(\delta_k, g_k)} =+\infty.
  \end{equation}
  Then, the sequence $\set{\Reg_{\Gamma(\delta_k,
  g_k)}(g_k)}_{k\in\N}$ is bounded and each weak cluster
  point is a $J$-minimizing solution of $Ku=g$. Additionally,
  with $\xi_k = K^* \Reg^*_{\Gamma(\delta_k,
  g_k)}(g_k) \in \partial J(\Reg_{\Gamma(\delta_k,
  g_k)}(g_k))$  it holds
  \begin{equation}\label{summary:conv_convfunc}
    \lim_{k\to \infty} J(\Reg_{\Gamma(\delta_k,
    g_k)}(g_k)) = J(u^\dagger)\ \text{ and }\ \lim_{k\to \infty}
    D^{\xi_k}_J(u^\dagger,\Reg_{\Gamma(\delta_k, g_k)}(g_k)) = 0,
  \end{equation}
  and the residuum satisfies the rate
  \begin{equation}\label{summary:conv_convimg}
    \norm{K\Reg_{\Gamma(\delta_k,
    g_k)}(g_k) - g} = \bigo(t_{\Gamma(\delta_k,
    g_k)}^{-1\slash 2}).
  \end{equation}
\end{theorem}
As indicated in Section \ref{defs:sc}, the speed of convergence in
\eqref{summary:conv_convfunc} can be arbitrarily slow, unless one imposes regularity restrictions on the
true solutions of $Ku=g$.
 We recall below Theorem 6.3 from \cite{FriSch10} in this respect.

\begin{theorem}\label{summary:ratesest} 
  Assume that the stopping rule $\Gamma:(0,\infty)\times H_2
  \ra \N$ satisfies $\lim_{k\ra\infty}  t_{\Gamma(\delta_k, g_k)} = +\infty$. Then the following two conditions are equivalent: 
  
  (i) There exists a $J$-minimizing solution $u^\dagger$ of $Ku=g$ that satisfies the source
  condition \eqref{defs:sceqn} with source element
  $p^\dagger\in H_2$ and there exists $C\in\R$ such that
  \begin{equation}\label{bound_strong}
    \delta_k t_{\Gamma(\delta_k, g_k)}\leq C.
  \end{equation}
  
  (ii) For $k\ra\infty$, one has 
  \begin{equation}\label{summary:conv_convimg1}
    \norm{K\Reg_{\Gamma(\delta_k,
        g_k)}(g_k) - g} = \bigo(t_{\Gamma(\delta_k,
      g_k)}^{-1})\ \text{ and }\ \left\|\Reg^*_{\Gamma(\delta_k,
        g_k)}(g_k)\right\|=\bigo(1).
  \end{equation}
  Additionally, if (i) or (ii) holds, then 
  \begin{equation*}\label{summary:conv_convfunc1}
    D^{K^*p^\dagger}_J(\Reg_{\Gamma(\delta_k, g_k)}(g_k),u^\dagger) = \bigo(t_{\Gamma(\delta_k,
      g_k)}^{-1})
  \end{equation*}
  and each cluster point of $\left\{\Reg^*_{\Gamma(\delta_k,
      g_k)}(g_k)\right\}$ is a minimizer of $G(\cdot,g)$.
\end{theorem}

 Theorem \ref{summary:ratesest2} and
Corollary \ref{summary:ratesestcor} below provide quantitative estimates for the
primal and dual iterates of the ALM in case that the source condition \eqref{defs:sceqn} holds. These
results extend \cite[Thm. 6.2]{FriSch10}.

\begin{theorem}\label{summary:ratesest2}
  Assume that $u^\dagger$ is a
  $J$-minimizing solution of $Ku=g$ which satisfies the source
  condition \eqref{defs:sceqn} with source element
  $p^\dagger\in H_2$. Then, for any  $\gamma > 0$
  \begin{equation}\label{summary:ratesesteqn1}
      D^{u^\dagger}_{J^*}(K^*p_n^\delta, K^*p^\dagger) +\frac{t_n}{4}
      \norm{Ku_n^\delta - g}^2 + \frac{\gamma - 1}{2\gamma t_n}
      \norm{p_n^\delta - p^\dagger}^2 \leq
      \frac{\norm{p^\dagger - p_0^\delta}^2}{2 t_n}+ \frac{(1+\gamma) t_n}{2} \delta^2.
  \end{equation}
\end{theorem}

\begin{proof}  Since $u^\dagger$ satisfies the source condition, we have that
$K^*p^\dagger\in \partial J(u^\dagger)$ which is equivalent to $u^\dagger \in
\partial J^*(K^* p^\dagger)$. This leads to
\begin{eqnarray*}
G(p_n^\delta,g^\delta) -G(p^\dagger,g^\delta)&=& G(p_n^\delta,g)
-G(p^\dagger,g) +\inner{p_n^\delta-p^\dagger}{g-g_k}\\ &=&
J^*(K^*p_n^\delta)-J^*(K^*p^\dagger)-\inner{p_n^\delta-p^\dagger}{g}+\inner{p_n^\delta-p^\dagger}{g-g^\delta}\\
&=&
J^*(K^*p_n^\delta)-J^*(K^*p^\dagger)-\inner{K^*p_n^\delta-K^*p^\dagger}{u^\dagger}+\inner{p_n^\delta-p^\dagger}{g-g^\delta}\\
&=&
D_{J^*}^{u^\dagger}(K^*p_n^\delta,K^*p^\dagger)+\inner{p_n^\delta-p^\dagger}{g-g^\delta}.
\end{eqnarray*}
Therefore, the last inequality together with Proposition \ref{summary:gueler}
and Young's inequality gives for an arbitrary $\gamma > 0$
\begin{multline}\label{summary:bregaux1}
 D^{u^\dagger}_{J^*}(K^*p_n^\delta, K^*p^\dagger) = G(p_n^\delta,g^\delta)
 -G(p^\dagger,g^\delta) + \inner{p_n^\delta-p^\dagger}{g^\delta-g} \\
 \leq \frac{\norm{p^\dagger - p_0^\delta}^2}{2 t_n} -
 \left(\frac{\gamma-1}{\gamma}\right)\frac{\norm{p^\dagger - p_n^\delta}^2}{2
 t_n} - \frac{t_n \norm{p^\delta_n - p_{n-1}^\delta}^2}{2 \tau_n^2} + \frac{\gamma}{2}\delta^2 t_n
\end{multline}
Using \eqref{intro:almupdate} together with the inequality
$\norm{Ku_n^\delta - g}^2 \leq 2 \norm{Ku_n^\delta - g^\delta}^2+2\delta^2$
and the previous estimate show the assertion.\qquad
\end{proof}

\begin{lemma}\label{summary:auxlem}
Let $a,b > 0$. Then, 
\begin{equation*}
  \inf_{\gamma > 1} \left(\frac{\gamma}{\gamma-1}a +
  \frac{\gamma^2}{\gamma-1}b\right) = \left(\sqrt{b} + \sqrt{a+b}\right)^2.
\end{equation*}
\end{lemma}

\begin{proof}
  With elementary calculus it is straightforward to deduce that the function
  $f(\gamma) = (\gamma-1)^{-1}(\gamma a + \gamma^2 b)$ attains its minimum among all
  $\gamma>1$ at
  \begin{equation*}
      \gamma_* = 1 + \sqrt{1 + \frac{a}{b}}.
  \end{equation*}
  Then, $\gamma_*\slash(\gamma_* - 1) = (\sqrt{b} +
  \sqrt{a+b})\slash\sqrt{a+b}$ and hence
  \begin{equation*}
  f(\gamma^*) =  \frac{\gamma_*}{\gamma_* -1}\left( a + \gamma_* b\right)  = 
  \frac{\gamma_*}{\gamma_* -1}\left(a+b + \sqrt{b}\sqrt{a+b}\right) =
  \left(\sqrt{b} + \sqrt{a+b}\right)^2.
  \end{equation*}
\end{proof}

\begin{corollary}\label{summary:ratesestcor} 
Let the assumptions of Theorem \ref{summary:ratesest2} hold.
\begin{enumerate}[i)]
  \item If $0< \alpha < 1\slash 2$, then
\begin{equation*}\label{summary:bregdistsest}
  \alpha D^{\text{sym}}_J(u_n^\delta, u^\dagger) +
  D^{u^\dagger}_{J^*}(K^*p_n^\delta, K^*p^\dagger) \leq 
  \frac{1-\alpha}{1-2\alpha} \delta^2 t_n + \frac{\norm{p^\dagger -
  p_0^\delta}^2}{2 t_n}.
\end{equation*}  
\item  It holds
\begin{equation*}\label{summary:bregdistest}
  D^{\text{sym}}_J(u_n^\delta, u^\dagger) \leq \norm{Ku_n^\delta -
  g}\Bigl(\delta t_n + \sqrt{\delta^2 t_n^2 + \norm{p_0^\delta
  -p^\dagger}^2}\Bigr).
\end{equation*}
\end{enumerate}
\end{corollary}

\begin{proof} 
  From Young's inequality it follows that
  \begin{equation*}
  \alpha D_J^\text{sym}(u_n^\delta, u^\dagger)=\alpha \inner{p_n^\delta-
  p^\dagger}{Ku_n^\delta-g} \leq \frac{\alpha}{t_n} \norm{p_n^\delta -
  p^\dagger}^2 + \frac{t_n}{4} \norm{Ku_n^\delta -g}^2.
  \end{equation*}
  Hence the first inequality follows from Theorem \ref{summary:ratesest2} with
  $\gamma=1\slash (1-2\alpha)$, due to the fact that  $\alpha < 1\slash 2$.
  
  In order to prove ii) we observe from \eqref{summary:bregaux1} that for all
  $\gamma > 1$
  \begin{equation*}
      \norm{p_n^\delta - p^\dagger}^2 \leq \frac{\gamma}{\gamma-1}
      \norm{p_0^\delta - p^\dagger}^2 + \frac{\gamma^2}{\gamma-1}\delta^2t_n^2.
  \end{equation*}
  Hence,  Lemma \ref{summary:auxlem} with $a = \norm{p_0^\delta
  - p^\dagger}^2$ and $b=\delta^2t_n^2$ leads to
  \begin{equation*}
       \norm{p_n^\delta - p^\dagger}^2 \leq \left( \delta t_n +  \sqrt{\delta^2
       t_n^2 + \norm{p_0^\delta -p^\dagger}^2}\right)^2.
  \end{equation*}
  Finally, the assertion follows from
  \begin{equation*}
      D^{\text{sym}}_J(u_n^\delta, u^\dagger) = \inner{p_n^\delta -
      p^\dagger} {K(u_n^\delta - u^\dagger)}\leq \norm{Ku_n^\delta -
      g}\norm{p_n^\delta - p^\dagger}.
  \end{equation*}
\end{proof}

\begin{remark}\label{summary:remonrates}\upshape
\begin{enumerate}[i)]
  \item Obviously, the best possible rates with respect to the
estimates in Theorem \ref{summary:ratesest} and Corollary \ref{summary:ratesestcor} i) are obtained
when $t_{\Gamma(\delta, g^\delta)} \sim \delta^{-1}$. However, if one only has
\begin{equation*}\label{summary:paramchoice}
  \delta t_{\Gamma(\delta, g^\delta)} \leq  C,
\end{equation*} 
for some $C>0$, then Corollary \ref{summary:ratesestcor} ii) shows that the
symmetric Bregman distance behaves at least as well as the residual:
\begin{equation*}
  D^{\text{sym}}_J(u_n^\delta, u^\dagger) = \bigo\left(\norm{Ku_n^\delta -
  g}\right). 
\end{equation*}  
\item Since $K^*p_n^\delta\in\partial J(u_n^\delta)$ and $K^*p^\dagger \in
\partial J(u^\dagger)$ is equivalent to $u_n^\delta\in \partial
J^*(K^*p_n^\delta)$ and $u^\dagger\in\partial J^*(p^\dagger)$ respectively, it
follows that
\begin{eqnarray*}
 D^{\text{sym}}_{J^*}(K^*p_n^\delta, K^*p^\dagger) &
 = & D^{u^\dagger}_{J^*}(K^*p_k, K^*p^\dagger)+
 D^{u_k}_{J^*}(K^*p^\dagger,K^*p_k ) \\
 & = &  \inner{u_k-u^\dagger}{K^*p_k-K^*p^\dagger}\\ 
 &=& D^{\text{sym}}_J(u_n^\delta, u^\dagger).
\end{eqnarray*}
Hence, all estimates for the primal variables $\set{u_n^\delta}_{n\in\N}$
automatically  hold also for $\set{K^*p_n^\delta}_{n\in\N}$.
\end{enumerate}
\end{remark}

%%%%%%%%%%%%%%%%%%%%%%%%%%%%%%%%%%%%%%%%%%%%%%%%%%%%%%%%%%%%%%%%%%%%%%%%%%%%%%%%%%%%
\section{Morozov's discrepancy principle}\label{morozov}

In this section we analyze the discrepancy principle as an a posteriori stopping rule.
In order to apply the convergence (rate) results in
Theorems~\ref{summary:conv},~\ref{summary:ratesest},
and~\ref{summary:ratesest2}, a given stopping rule
$\Gamma:(0,\infty) \times H_2\ra \N$  has to satisfy \eqref{summary:abbrev} and
\eqref{bound_strong}, respectively. We verify these estimates for the
particular situation where the stopping index is chosen according to
\emph{Morozov's discrepancy principle}: Choose $\rho > 1$ and define
\begin{equation}\label{morozov:disc}
  \Gamma(\delta, g^\delta) := \min\set{n\in\N~:~ \norm{K u_n^\delta - g^\delta}
  < \rho \delta}.
%  \Gamma(\delta, g^\delta) := \max\set{n\in\N~:~ \norm{K u_n^\delta -g^\delta} > \rho \delta}.
\end{equation}
That is, we take the first iterate $u_n^\delta$ for which the residual
$\norm{K u_n^\delta - g^\delta}$ falls below a number which is a constant $\rho$ times the noise level $\delta$.

\begin{proposition}
The stopping rule \eqref{morozov:disc} is well defined.
\end{proposition}

\begin{proof}
  It follows from \cite[Cor.~5.2]{FriSch10} that there exists a constant $C >
  0$ such that
  \begin{equation*}
    \frac{1}{2}\norm{K u_n^\delta - g^\delta}^2 \leq \frac{C}{t_n} +
    \frac{\delta^2}{2}
  \end{equation*}
  This implies that for all $\rho > 1$ there exists an index $n_0\in
  \N$ for which $\norm{Ku_{n_0}^\delta-g^\delta}<\rho\delta$. 
  Thus, $\Gamma(\delta, g^\delta)<\infty$ is ensured.\qquad
\end{proof}

Our analysis is structured as follows: In Section \ref{morozov:qualitative}, we
derive convergence rates (based on Corollary \ref{summary:ratesestcor} ii)) for
the symmetric Bregman-distance between the primal iterates $\set{u_n^\delta}_{n\in\N}$ and $J$-minimizing solutions of $Ku=g$, under the
hypothesis that the source condition holds. Here, we make no other assumption
on $\Gamma(\delta, g^\delta)$ except \eqref{morozov:disc}. 
In Section \ref{sec:nondegenerate-case} we then point out that the convergence
results in Theorems~\ref{summary:conv} and~\ref{summary:ratesest} apply for the
parameter choice rule \eqref{morozov:disc} if additionally one requires $\lim_{\delta\ra0} 
{\Gamma(\delta, g^\delta)} = \infty$. We refer to this situation as the
\emph{non-degenerate case}.
Finally in Section \ref{sec:degenerate-case} we treat the \emph{degenerate
case}, i.e., where $\{\Gamma(\delta, g^\delta)\}_{\delta}$ has finite
accumulation points.  

\subsection{Convergence rates.}\label{morozov:qualitative}

We will state a qualitative estimate for the Bregman distance between the primal
variables in the ALM  and solutions of \eqref{intro:constrained} if
the source condition is satisfied and if the Morozov stopping rule
is applied. In particular, this analysis sheds some light
on the role of $\rho$ in \eqref{morozov:disc}.

\begin{lemma}\label{morozov:pararulesc}
Let $u^\dagger$ be a $J$-minimizing
solution of $Ku=g$ that satisfies the source condition with source element
$p^\dagger$ and assume that $\Gamma$ is chosen according to the stopping
  rule~\eqref{morozov:disc}. Then, 
\begin{equation}\label{morozov:sourcecondparamest}
    \delta t_{\Gamma(\delta, g^\delta)} \leq \frac{\norm{p^\dagger -
    p_0^\delta}}{\sqrt{\rho -1}} + \delta \bar\tau.
\end{equation}
In particular, \eqref{bound_strong} is satisfied. 
\end{lemma}  
\begin{proof}
  Let $g^\delta \in H_2$ and set $\delta:= \norm{g - g^\delta}$ as well as $n_* =
  \Gamma(\delta, g^\delta)-1$.  Then, it follows from  \eqref{morozov:disc} that
  \begin{equation*}
      \norm{K u_{n^*}^\delta - g^\delta} \geq \rho \delta.
  \end{equation*}
  This together with \eqref{summary:guelest} yields
  \begin{equation*}
    \frac{\norm{p - p_{n_*}^\delta}^2}{2 t_{n_*}} +  \frac{\rho \delta^2
    t_{n_*}}{2} \leq G(p,g^\delta) - G(p_{n_*}^\delta, g^\delta) +  \frac{\norm{p -
    p_0^\delta}^2}{2 t_{n_*}}
  \end{equation*}
  for all $p \in H_2$ (recall that  $\norm{K u_n^\delta - g^\delta} =
  \tau_n^{-1}\norm{p_n^\delta - p_{n-1}^\delta}$ by
  \eqref{intro:almupdate}). From the definition of $G$ it follows that $G(p,g^\delta) - G(p_{n_*}^\delta, g^\delta) = G(p,g) - G(p_{n_*}^\delta, g) + \inner{p - p_{n_*}^\delta}{g - g^\delta}$. After applying Young's inequality to the inner product we get, for every $p\in H_2$ and $\eta>0$,
  \begin{equation*}
    \frac{\norm{p - p_{n_*}^\delta}^2}{2 t_{n_*}} + \frac{\rho \delta^2 t_{n_*}}{2} \leq G(p,g) - G(p_{n_*}^\delta, g) + \frac{\norm{p - p_{n_*}^\delta}^2}{2\eta} + \frac{\eta\norm{g - g^\delta}^2}{2} + \frac{\norm{p - p_0^\delta}^2}{2 t_{n_*}}.
  \end{equation*}
  Setting $\eta = t_{n_*}$ hence gives
  \begin{equation}\label{morozov:est}
    \frac{(\rho-1) \delta^2 t_{n_*}}{2} \leq G(p,g) - G(p_{n_*}^\delta, g) +
    \frac{\norm{p - p_0^\delta}^2}{2 t_{n_*}}.
  \end{equation}
  Since $u^\dagger$ satisfies the source condition with source element
  $p^\dagger$, it follows from \cite[Prop. 6.1]{FriSch10} that $G(p^\dagger,g)
  \leq G(p,g)$ for all $p\in H_2$. Moreover, using $p^\dagger$ instead of $p$ in \eqref{morozov:est} shows
  \begin{equation*}
    \frac{(\rho-1) \delta^2 t_{n_*}}{2} \leq   \frac{\norm{p^\dagger -
    p_0^\delta}^2}{2 t_{n_*}}
  \end{equation*}
  or in other words
  \begin{equation*}
    \delta t_{\Gamma(\delta, g^\delta)} \leq \frac{\norm{p^\dagger -
    p_0^\delta}}{\sqrt{\rho -1}} + \delta \bar\tau.
  \end{equation*}
\end{proof}

With this preparation we are ready to state the announced estimate for the
primal variables.
 
\begin{theorem}\label{morozov:goodest}  Let the assumptions of Lemma
\ref{morozov:pararulesc} be satisfied. Then,
\begin{equation}\label{morozov:goodestimate}
       D_J^{\text{sym}}(\Reg_{\Gamma(\delta, g^\delta)}(g^\delta), u^\dagger)
       \leq \left(1 +
       \bigo(\sqrt{\delta})\right)\frac{\rho(\sqrt{\rho}+1)}{\sqrt{\rho-1}}
       \norm{p_0^\delta - p^\dagger} \delta.
\end{equation}
\end{theorem}

\begin{proof}  
 From \eqref{morozov:sourcecondparamest} it follows that
 \begin{eqnarray*}
    \delta^2 t_{\Gamma(\delta, g^\delta)}^2 + \norm{p^\dagger -
    p_0^\delta}^2& \leq & \frac{1}{\rho - 1}\norm{p^\dagger -
    p_0^\delta}^2 + \frac{2\delta \bar\tau}{\sqrt{\rho - 1}}\norm{p^\dagger -
    p_0^\delta}  + \delta^2 \bar\tau^2 +  \norm{p^\dagger -
    p_0^\delta}^2\\
    & = & \frac{\rho}{\rho - 1}\norm{p^\dagger -
    p_0^\delta}^2 + \delta\left(\frac{2 \bar \tau}{\sqrt{\rho -
    1}}\norm{p^\dagger - p_0^\delta} + \delta \bar\tau^2\right) \\
    & = & \frac{\rho}{\rho - 1}\norm{p^\dagger - p_0^\delta}^2 + \bigo(\delta).
 \end{eqnarray*}
 This together with \eqref{morozov:sourcecondparamest} and the fact that 
$\sqrt{a+b}\leq\sqrt{a} + \sqrt{b}$ for all $a,b>0$ implies
 \begin{equation*}
    \delta t_{\Gamma(\delta, g^\delta)} + \sqrt{ \delta^2 t_{\Gamma(\delta,
    g^\delta)}^2 + \norm{p^\dagger
    - p_0^\delta}^2 } \leq \frac{\sqrt{\rho} + 1}{\sqrt{\rho -
    1}}\norm{p^\dagger - p_0^\delta} + \bigo(\sqrt{\delta}).
 \end{equation*}
 Since by construction in \eqref{morozov:disc} 
 \begin{equation*}
    \norm{K\Reg_{\Gamma(\delta, g^\delta)}(g^\delta) - g^\delta} < \rho\delta,
 \end{equation*}
 the assertion follows from Corollary \ref{summary:ratesestcor} ii).\qquad
\end{proof}
  
\begin{remark}\label{morozov:goodestrem}\upshape
    The function 
    \begin{equation*}
          f(\rho) := \frac{\rho(\sqrt{\rho}+1)}{\sqrt{\rho-1}}
    \end{equation*} 
    which appears in the right hand side of \eqref{morozov:goodestimate} is
    minimal for $\rho^*\simeq 1.6404$ with $f(\rho^*)\simeq 4.6753$. Hence,
    after setting $\rho = \rho^*$ in the stopping rule \eqref{morozov:disc},
    Theorem \ref{morozov:goodest} implies the following rough estimate
    \begin{equation*}
      D_J^{\text{sym}}(\Reg_{\Gamma(\delta, g^\delta)}(g^\delta), u^\dagger)
      < 5\norm{p_0^\delta - p^\dagger}\delta
    \end{equation*}
    as $\delta\ra 0$.
\end{remark}

\subsection{The nondegenerate case}\label{sec:nondegenerate-case}

In this section we will show that the assumptions of Theorems
\ref{summary:conv} and \ref{summary:ratesest} are satisfied for the stopping
rule \eqref{morozov:disc}, if additionally one requires
\begin{equation}\label{morozov:stoppinfty}
  \lim_{k\ra\infty}\Gamma(\delta_k, g_k) = \infty.
\end{equation} 
From Lemma \ref{morozov:pararulesc} it already follows that
\eqref{bound_strong} holds which implies applicability of Theorem
\ref{summary:ratesest}. Moreover, we find 

\begin{lemma}\label{morozov:convergence}
  Assume that $\Gamma$ is chosen according to the stopping
  rule~\eqref{morozov:disc} and that \eqref{morozov:stoppinfty} holds. 
  Then, $\Gamma(\delta_k,g_k)$ satisfies \eqref{summary:abbrev},
  i.e.~$\delta_k^2t_{\Gamma(\delta_k,g_k)}\to 0$ and $t_{\Gamma(\delta_k,g_k)}\to+\infty$, as $k\to\infty$.  
\end{lemma}

\begin{proof}
  Let $\eps> 0$ and choose $p_\eps \in H_2$ such that $G(p_\eps, g)
  \leq \inf_{q\in H_2} G(q,g) + \eps$ (note that, due to \cite[Lem. 4.1]{FriSch10}, the
  right hand side is finite whenever $g$ is attainable). This together with the estimate
  \eqref{morozov:est} in the proof of Theorem \ref{morozov:goodest} shows
  \begin{equation*}
    \frac{(\rho-1) \delta^2 t_{n_*}}{2} \leq \eps + \frac{\norm{p_\eps - p_0^\delta}^2}{2 t_{n_*}}.
  \end{equation*}
According to \eqref{infty}, the conditions $\tau_k \leq \bar \tau$ for all $k\in\N$, and 
$\lim_{k\ra\infty}  {\Gamma(\delta_k, g_k)} = \infty$ imply
$\lim_{k\ra\infty} t_ {\Gamma(\delta_k, g_k)} =+\infty$. Hence, substituting $g_k$ for $g^\delta$, $\delta_k$ for $\delta$, and
  $\Gamma(\delta_k, g_k)-1$ for $n_*$ shows
  \begin{equation*}
    \limsup_{k\ra\infty} \delta_k^2 t_{\Gamma(\delta_k, g_k)} \leq
    \limsup_{k\ra\infty}\left( \frac{2\eps}{\rho -1} + \frac{\norm{p_\eps -
    p_0^\delta}^2}{2 t_{\Gamma(\delta_k, g_k)-1}} + \delta_k^2 \bar\tau \right)=
    \frac{2\eps}{\rho -1}.
  \end{equation*}
  Since $\eps$ is arbitrary, this proves the statement.\qquad
\end{proof}

Combining the above results with Theorem \ref{summary:conv} yields
results on convergence for Morozov's discrepancy principle
as a stopping rule:

\begin{corollary}\label{cor_conv} Assume that $\Gamma$ is chosen as in
\eqref{morozov:disc}  and that \eqref{morozov:stoppinfty} holds. Then, the
  sequence $\set{\Reg_{\Gamma(\delta_k, g_k)}(g_k)}_{k\in\N}$ is bounded and each
  weak cluster point $u^{\dagger}$ is a $J$-minimizing solution of $Ku=g$.
  Additionally, \eqref{summary:conv_convfunc} and \eqref{summary:conv_convimg}  hold.
\end{corollary}

If additionally the source condition is satisfied, Lemma
\ref{morozov:convergence} and Theorem \ref{summary:ratesest} imply

\begin{corollary}\label{cor_conv_2} Let the assumptions of Corollary \ref{cor_conv}
be satisfied and assume that there exists a solution $u^\dagger$ of
\eqref{intro:constrained} which verifies the source condition with source
element $p^\dagger$. Then, \eqref{summary:conv_convimg1} holds and each weak cluster point of 
    $\set{\Reg^*_{\Gamma(\delta_k, g_k)}(g_k)}_{k\in\N}$ is a minimizer of
    $G(\cdot,g)$.
\end{corollary}

\begin{remark}\label{morozov:cor_conv_better_rem}\upshape
 From Schauder's Theorem and from $\overline{\ran(K)}
    = \ker(K^*)^\bot$ it follows that for each compact $K$ with dense range,
    the adjoint operator $K^*$ is compact and injective and hence
    \begin{equation*}
          \lim_{k\ra\infty} K^*\Reg^*_{\Gamma(\delta_k, g_k)}(g_k) = K^*\bar p
    \end{equation*} 
    strongly, where $\bar p$ is a minimizer of $G(\cdot,g)$. If the condition
    on the range of $K$ is not satisfied, then strong convergence hold on
    subsequences. 
\end{remark}

\subsection{The degenerate case}\label{sec:degenerate-case}

We will finally discuss the case when the stopping index chosen by Morozov's
discrepancy principle degenerates, that is, when there exists an $N \in \N$ such
that
\begin{equation}\label{morozov:deg}
  \limsup_{\delta \ra 0^+} \Gamma(\delta, g^\delta) = N.
\end{equation}
In this case, the assumption \eqref{morozov:stoppinfty} is not satisfied and
the results of Section \ref{sec:nondegenerate-case} do not apply in general.

The following result shows, however, that a degenerate stopping rule as in
\eqref{morozov:deg} already implies that the true solutions of
\eqref{intro:constrained} satisfy the source condition \eqref{defs:sceqn} and
hence the results in Section \ref{morozov:qualitative} hold.  Moreover, the
convergence (on subsequences) of the dual sequence also follows. 

\begin{theorem}\label{morozov:degconv}
  Let $\Gamma:(0,\infty)\times H_2 \ra \N$ be as in
  \eqref{morozov:disc} and assume that \eqref{morozov:deg} holds.  Then, the
  following assertions are true:
  \begin{enumerate}[i)]
    \item The set $\set{p_N^\delta}_{\delta > 0}$ is bounded and each of its weak cluster points is a minimizer of $G(\cdot,g)$.
    \item  The set $\set{u_N^\delta}_{\delta > 0}$ is bounded
    and each of its weak cluster points is a $J$-minimizing
    solution of $Ku=g$.
    \item All $J$-minimizing solutions of $Ku=g$ satisfy the source
    condition with a source element $p^\dagger$.
    \item
    \begin{equation}\label{morozov:degconv_convfunc}
      \norm{K u_N^\delta- g} < (\rho + 1)\delta\quad\text{ and }\quad
      D_J^\text{sym}(u^\delta_N, u^\dagger) = \bigo(\delta).
    \end{equation}
  \end{enumerate}
\end{theorem}

\begin{proof}
  The
  definition of $\Gamma(\delta, g^\delta)$ in \eqref{morozov:disc} and the monotonicity of the residual $\norm{Ku_n^\delta - g^\delta}$ (cf. \cite[Cor. 3.3]{FriSch10}) imply
  \begin{equation}\label{morozov:degpr1}
    \norm{ K u_N^\delta - g^\delta} \leq \rho \delta,\quad \text{ for all }\delta
    > 0.
  \end{equation}
  In particular, this yields $\norm{ K u_N^\delta - g} \leq (\rho+1)\delta$.
  
  It was shown in the proof of \cite[Thm 5.3]{FriSch10} (by using
  G\"uler's estimate \eqref{summary:guelest} and Young's inequality) that
  \begin{equation*}
    \norm{p - p_N^\delta}^2 \leq 2\norm{p - p_0^\delta}^2 +
    4 t_N^2 \delta^2 + 4t_N(G(p,g) - \inf_{q \in H_2} G(q,g))\quad \text{ for all
    }p\in H_2.
  \end{equation*}
  Choosing an arbitrary $p$ such that $G(p,g) < 0$ implies   
  \begin{equation*}\label{morozov:degdualbnd}
    \limsup_{\delta \ra 0^+} \norm{p_N^\delta} =: A < \infty.
  \end{equation*}
  Now, let $\set{\delta_k}_{k\in\N}$ be such that $\delta_k \ra 0^+$ and that
  $p_N^{\delta_k} \rightharpoonup \hat p \in H_2$. Due to the dual
  characterization \eqref{summary:dual} and to the equality  $p_{N-1}^\delta -
  p_N^\delta = \tau_N (K u_N^\delta - g^\delta)$, it follows that $K
  u^{\delta_k}_N - g^{\delta_k} \in \partial G(\cdot,
  g^{\delta_k})(p_N^{\delta_k})$. Since $G(p,g) = G(p,g^{\delta_k}) +
  \inner{p}{g^{\delta_k} - g}$ for all $p\in H_2$, one has
  \begin{equation*}
    K u^{\delta_k}_N - g \in \partial G(\cdot, g)(p_N^{\delta_k}).
  \end{equation*}
  Recall that the graph of the subgradient of a convex and lower semi-continuous
  functional is weakly-strongly closed. Therefore, inequality 
  \eqref{morozov:degpr1} yields
  \begin{equation*}
    0 = \lim_{k\ra\infty} K u^{\delta_k}_N - g \in \partial G(\cdot,
    g)(\wlim_{k\ra\infty} p_N^{\delta_k}) = \partial G(\cdot, g)(\hat p).
  \end{equation*}
  This proves i).
  
  From the definition of $u^\delta_N$ in \eqref{intro:almobjective}
  and the fact that $p_{N-1}^\delta - p_N^\delta = \tau_N (K u_N^\delta -
  g^\delta)$ it follows (for $\delta$ small enough)
  \begin{equation*}
    \begin{split}
      \frac{\tau_N}{2}\norm{K u^\delta_N - g^\delta} + J(u^\delta _N) & \leq
      \frac{\tau_N}{2}\delta^2 + J(u^\dagger) + \inner{p_{N-1}^\delta}{g - K
      u^\delta_N} \\
      & \leq \frac{\tau_N}{2}\delta^2 + J(u^\dagger) + A(\rho+1)\delta +
      \tau_N \rho(\rho+1)\delta^2.
    \end{split}
  \end{equation*}
  In other words, $J(u^\delta_N) - J(u^\dagger)  = \bigo(\delta)$ as $\delta \ra
  0^+$. This together with \eqref{morozov:degpr1} shows that
  $\sup_{\delta > 0} \left\{J(u^\delta_N) + \norm{Ku_N^\delta}\right\} < \infty$ and
  consequently, according to Assumption \ref{defs:compactnessass}, that
  $\set{u_N^\delta}_{\delta>0}$ is weakly compact and hence bounded. Thus, ii) follows from \eqref{morozov:degpr1} and the lower semi-continuity of $J$.
  
  Let $p^\dagger$ be a minimizer of $G(\cdot,g)$, which exists according to i).
  This and the definition of $G(p,g)$ in \eqref{summary:descent} implies
  \begin{equation*}
    G(p^\dagger, g) - G(p_N^\delta, g^\delta) \leq \delta \norm{p_N^\delta}.
  \end{equation*}
  Moreover, we deduce from the optimality condition of \eqref{intro:almobjective} that
  $K^*p_N^\delta \in \partial J(u_N^\delta)$, which in turn implies that
  $Ku_N^\delta \in \partial (J^*\circ K^*)(p_N^\delta)$. Using the definition of
  the subgradient and some rearrangements give
  \begin{equation*}
    G(p^\dagger, g) - G(p_N^\delta, g^\delta) \geq - \delta
    \left(\norm{p^\dagger} + \norm{p_N^\delta}\right).
  \end{equation*}
  Since $\set{\norm{p_N^\delta}}_{\delta > 0}$ is bounded according to ii), the
  previous two estimates result in
  \begin{equation}\label{morozov:degpr2}
    \lim_{\delta\ra 0^+} J^*(K^*p_N^\delta) - \inner{p_N^\delta}{
    g^\delta}  = \lim_{\delta\ra 0^+} G(p_N^\delta, g^\delta) = G(p^\dagger, g) =
    J^*(K^*p^\dagger) - \inner{p^\dagger}{g}.
  \end{equation}
  Using once more the relation $K^*p_N^\delta \in \partial J(u_N^\delta)$ shows
  that $J^*(K^* p_N^\delta) + J(u_N^\delta) = \inner{K^* p_N^\delta}{u_N^\delta}$
  and consequently
  \begin{equation*}
    J^*(K^* p_N^\delta) - \inner{p_N^\delta}{g^\delta} + J(u_N^\delta) =
    \inner{Ku_N^\delta - g^\delta}{p_N^\delta}.
  \end{equation*}
  Now, let $u^\dagger$ be a $J$-minimizing solution of $Ku=g$ which
  exists according to ii). Taking the limit $\delta\ra 0^+$ in the previous
  equality, using \eqref{morozov:degpr1},
  \eqref{morozov:degpr2}, as well as the boundedness of $\set{p_N^\delta}_{\delta>
  0}$ and the fact that $J(u_N^\delta) \ra J(u^\dagger)$ result in
  \begin{equation*}
    J(u^\dagger) + J^*(K^* p^\dagger) = \inner{p^\dagger}{g} =
    \inner{K^*p^\dagger}{u^\dagger}
  \end{equation*}
  that is, $K^*p^\dagger \in \partial J(u^\dagger)$. This proves iii).
  
  Statement iv)  follows from i), iii) and Corollary
  \ref{summary:ratesestcor} ii) together with the first inequality in
  \eqref{morozov:degconv_convfunc}.\qquad
\end{proof}

\begin{remark}\upshape
As  $\{\Gamma(\delta, g^\delta)\}_{\delta>0}$ has finite
accumulation points, without restricting generality, we can consider that this is a constant
subsequence. This yields that for all $\delta$ sufficiently small, one has to
stop the algorithm at the same iteration. 

A degenerate case is discussed for the
Landweber method for nonlinear equations in the book \cite[p. 284]{EngHanNeu96}.
It is shown there that $\lim_{\delta\ra 0} u^\delta_N=u_N$  where $u_N$ is the
$N$-th iterate in the exact data case and is a solution of the operator
equation as well. This means that in the exact data case the Landweber algorithm reaches
the solution after $N$ steps, with $N$ being the stopping index in the noisy data
case. 

For the ALM analyzed here, we could not show  that  $\lim_{\delta\ra
0} u^\delta_N=u_N$  where $u_N$ is the $N$-th iterate in the exact data case
because the implicit feature of the method makes the analysis more difficult.
However, we could establish that the accumulation points of 
$\{u^\delta_N\}_{\delta>0}$ are $J$ - minimizing solutions with additional
smoothness, i.e.,  satisfying the source condition.
\end{remark}

The results for the two cases are briefly summarized in the following corollary.

\begin{corollary}  Let $\Gamma:(0,\infty)\times H_2 \ra \N$ be  chosen according to Morozov's rule
  \eqref{morozov:disc}. Then, as $\delta\ra 0$, the stopping index $\Gamma$ either increases and leads to weak convergence of the ALM algorithm on subsequences   to solutions of the operator equation  or is constant, in which case the corresponding ALM iterates  converge weakly on subsequences to a solution of the equation  satisfying the source condition. 
\end{corollary}

%%%%%%%%%%%%%%%%%%%%%%%%%%%%%%%%%%%%%%%%%%%%%%%%%%%%%%%%%%%%%%%%%%%%%%%%%%%%%%%%%%%%%%%%%%%%%%%
\section{Iterative total variation regularization}\label{sec:tv-regul}

The ALM method in the case of $J$ being the total variation
seminorm~\eqref{intro:tv} is also known as Bregman iteration
\cite{OshBurGolXuYin05}. It was shown in \cite{OshBurGolXuYin05} that Morozov's
discrepancy principle yields weak$^*$ convergence in $\BV$ of the
iterative method.  The expected but missing convergence  there was the one  with
respect to the total variation seminorm, in the sense
 \begin{equation}\label{j-conv}
\lim_{k\ra\infty} J(u_k)=J(u).
\end{equation}
  As a consequence of the analysis based on the augmented Lagrangian method tools, it became clear that this convergence does hold.  Moreover, linear convergence rates with respect to the Bregman distance associated with the total variation seminorm were established in \cite{BurResHe07} first for the noise free case.  According to \cite{BurOsh04} and due to the symmetric Bregman distance estimates pointed out in this work, such convergence rates provide information on the fine structure of the iterates, that is, the variation of the iterates is concentrated around the discontinuities set of the true solution. In the noisy data case, an a posteriori  stopping rule was proposed in \cite{OshBurGolXuYin05}:
\[
n_*(\delta,g^\delta)=\max\set{n\in\N: \|Ku_n^\delta-g^\delta\|\geq\rho\delta},
\quad\rho>1.  
\]
Although convergence was shown there for the net
$\{u_{n_*(\delta,g^\delta)}^\delta\}$ as $\delta\ra 0$, no convergence rate was obtained for it. This section aims to point out such a convergence rate. Note
that the a posteriori rule~\eqref{morozov:disc} employed here relates to the
above mentioned one by
\begin{equation*}
  \Gamma(\delta, g^\delta) = n_*(\delta, g^\delta)+1.
\end{equation*}

Still, the question on how to quantify the weak$^*$ convergence is not answered.
A possible answer could be given by taking into account that weak$^*$
convergence in $\BV$ together with convergence in the sense \eqref{j-conv} is
equivalent to so-called \textit{strict convergence}. Thus, one can obtain
convergence rates with respect to a related metric,  as shown below.
Recall \cite[page 125]{AmbFusPal00} that $\{u_ k\}_{k\in\N}\subset \BV$
converges \emph{strictly} to $u$ if it converges with respect to the metric
\begin{equation}\label{metric}
  \tilde d(u,v)=\|u-v\|_{\Ls{1}}+|J(u)-J(v)|.
\end{equation}
In this section we consider the linear and bounded operator $K:\L{2}\ra
\L{2}$,  where $\Omega\subset {\mathbb{R}}^2$ is open and bounded.
\begin{proposition}
  Let $\set{g_ k}_{ k\in\N} \subset \L{2}$ be such that
  $\norm{g-g_ k} \leq \delta_ k\ra 0$ as $ k\ra\infty$. Let $\Gamma$ be chosen according to the Morozov's rule \eqref{morozov:disc} and assume that $\lim_{
    k\ra\infty} \Gamma(\delta_ k, g_ k) = \infty$. Then, the sequence
  $\set{\Reg_{\Gamma(\delta_ k, g_ k)}(g_ k)}_{ k\in\N}$ satisfies 
  \eqref{summary:conv_convfunc} and \eqref{summary:conv_convimg}. Moreover, it
  has a subsequence which converges strictly to  a $J$-minimizing solution of $Ku=g$.
\end{proposition}

\begin{proof}
The first assertions result from Corollary \ref{cor_conv}. Let further denote
$u_k= \Reg_{\Gamma(\delta_ k, g_ k)}(g_ k)$. According to Corollary
\ref{cor_conv},  the sequence $\set{u_k}_{k\in\N}$  is bounded in $\L{2}$ and
$\sup_{k\in\N} J(u_k) < \infty$. Hence we find that
\begin{equation*}
  \sup_{k\in\N}\norm{u_k}_{\TV} = \sup_{k\in\N} \norm{u}_{\Ls{1}} + J(u_k) <
  \infty.
\end{equation*} 
Theorem 2.5 in \cite{AcaVog94} implies that $\set{u_k}_{k\in\N}$ is
strongly $\Ls{1}$-compact and thus there is a subsequence, indexed by
$k'$, which converges to some $u^*$ strongly in $\L{1}$. Since each $\Ls{2}$-weak
cluster point of $\set{u_k}_{k\in\N}$ is a $J$-minimizing solution of $Ku=g$ according
to Corollary \ref{cor_conv}, the same holds for $u^*$. Finally, it follows from
\eqref{summary:conv_convfunc} that $\tilde d(u_{k'}, u^*) \ra 0$.\qquad
\end{proof}

Clearly, error estimates in terms of the $\Ls{1}$-norm are desirable, but not easy to derive. In order to show convergence rates for strict convergence of the iterates,  we need to employ another metric, which appears naturally in the analysis, namely
\begin{equation}\label{equiv_metric}
  d(u,v)=\norm{Ku-Kv}_{\Ls{2}}+\abs{J(u)-J(v)}.
\end{equation}

The following lemma points out the relation between the two metrics described above.

\begin{lemma}
Assume that $K:\L{1}\ra \L{2}$ is continuous
and can be extended by continuity to $\L{2}$. Then, convergence of a sequence with respect to the metric $\tilde
d$ defined by \eqref{metric} implies convergence of the sequence with respect to the metric $d$
defined by \eqref{equiv_metric}. 
If additionally the linear  bounded operator $K:\L{2}\ra \L{2}$ is injective,
then the two metrics are equivalent.
\end{lemma}

\begin{proof} The first part follows immediately from $\norm{Ku}_{\Ls{2}}\leq \norm{K}\norm{u}_{\Ls{1}}$ for any $u\in\L{1}$. 

Assume now that $d(u_k,u)\ra 0$ as $k\ra\infty$ and that $K$ is injective. Then,
$K$ in particular does not annihilate constant functions and it follows from
\cite[Lemma 4.1]{AcaVog94} that $u\mapsto \norm{Ku}_{\Ls{2}} + J(u)$ is
$\TV$-coercive. Hence boundedness of $\set{\norm{Ku_k}_{\Ls{2}}}_{k\in\N}$ and
$\set{J(u_k)}_{k\in\N}$, which follows from $d(u_k,u)\ra 0$,  yields
boundedness of $\set{\norm{u_k}_{\TV}}_{k\in\N}$. Thus, there exists a subsequence  
$\set{u_{k'}}_{k'\in\N}$ which converges  to some $v\in \BV$ strongly in
$\L{1}$  and  weakly in $\L{2}$ to $v$ due to compact and bounded embedding
respectively (cf. \cite[Theorem 2.5]{AcaVog94}). These yield strong convergence
of the subsequence  in $\L{1}$  to $v$, as well as weak convergence in $\L{2}$ of $\set{Ku_{k'}}_{k'}$ to $Kv$. 
Since the weak limit is unique, it follows that  $Ku=Kv$ and consequently,
since $K$ is injective, that $u=v$.

Moreover, the entire sequence $\set{u_k}_{k\in\N}$ converges strongly in $\L{1}$
to $u$, which completes the proof.\qquad
\end{proof}
Note that the continuity of the operator $K$ from $\L{1}$ into $\L{2}$ is not necessary for proving the second part of the lemma. 

Now we are able to show the convergence rate in terms of the metric $d$:
\begin{proposition} 
  Let $\set{g_ k}_{ k\in\N} \subset \L{2}$ be such that
  $\norm{g-g_ k} \leq \delta_ k\ra 0$ as $ k\ra\infty$. Let $\Gamma$ be chosen according to rule \eqref{morozov:disc} and assume that $\lim_{
    k\ra\infty} \Gamma(\delta_ k, g_ k) = \infty$.
  If $u^\dagger$ is a $J$-minimizing
  solution of $Ku=g$ that satisfies the source condition
  \eqref{defs:sceqn}  with source element $p^\dagger\in H_2$, then the following convergence rate holds:
  \begin{equation*}\label{rate_metric}
    d(\Reg_{\Gamma(\delta_k,g_k)}(g_k),u^\dagger)=\|K\Reg_{\Gamma(\delta_k,g_k)}(g_k)-Ku^\dagger\|+|J(\Reg_{\Gamma(\delta_k,g_k)}(g_k))-J(u^\dagger)|=\bigo(\delta_k).
  \end{equation*}
\end{proposition}
\begin{proof}
  From the definition of rule \eqref{morozov:disc} it follows that
  $\|K\Reg_{\Gamma(\delta_k,g_k)}g_k-Ku^\dagger\|=\bigo(\delta_k)$.
  
  In order to establish an error estimate for
  $|J(\Reg_{\Gamma(\delta_k,g_k)}(g_k))-J(u^\dagger)|$, we use 
  Theorem~\ref{morozov:goodest}.
  Indeed, since the symmetric Bregman distance is larger than the Bregman
  distance, one has
  \begin{multline*}
    J(u^\dagger)-J(\Reg_{\Gamma(\delta_k,g_k)}(g_k))\leq \\
    \inner{\Reg^*_{\Gamma(\delta_k,g_k)}(g_k)}{g-K\Reg_{\Gamma(\delta_k,g_k)}(g_k)}
    + D^{\text{sym}}_J(\Reg_{\Gamma(\delta_k,g_k)}(g_k), u^\dagger).
  \end{multline*}
  Using the Cauchy-Schwarz inequality and again
  Corollary~\ref{summary:ratesestcor} we see that
  \[
  J(u^\dagger)-J(\Reg_{\Gamma(\delta_k,g_k)}(g_k)) = \bigo(\delta_k).
  \]
  Similarly one can show
  \[
  J(\Reg_{\Gamma(\delta_k,g_k)}(g_k))-J(u^\dagger) = \bigo(\delta_k)
  \]
  which ends the proof.\qquad
\end{proof}

%%%%%%%%%%%%%%%%%%%%%%%%%%%%%%%%%%%%%%%%%%%%%%%%%%%%%%%%%%%%%%%%%%%%%%%%%%%%%%%
\section{Sparse  regularization}\label{sec:sparse-regul}

In the case of sparse regularization,  the convex
functional~\eqref{intro:sparse}  is considered with $1\leq q\leq
2$ (see \cite{DauDefDeM04}). The aim of the functional $J$ is to promote sparse solutions,
i.e.~solutions which  have only a few (especially a finite number of) nonzero entries.
Tikhonov
regularization based on this regularization functional  has been studied in great detail
in~\cite{Lor08,GraHalSch08,LorSchTre10}. The case $q=1$ for the
stationary augmented Lagrangian method has been treated in~\cite{BurResHe07}
also under the name Bregman iteration. There, the authors
obtained convergence of the method for noisefree data for the Bregman
distance and considered an a priori stopping rule for noisy data. In
this section we also treat the case $q=1$ and derive both an enhanced
convergence rate for noisefree data in norm and also optimal
convergence rates for noisy data with the a posteriori rule given by
Morozov's discrepancy principle.

\subsection{Convergence rates for $\delta\to 0$}
\label{sec:conv-rates-delt}

We start with a result on convergence in the noisy data case  which holds for all $q\in[1,2]$. Fulfillment of a 
source condition is not needed here.
\begin{theorem}
  Let $K:\ell^2\to H_2$ be linear and bounded, $1\leq q\leq 2$ and let $J$ be defined by~\eqref{intro:sparse}. 
Moreover, let the parameter choice $\Gamma$
obey~\eqref{summary:abbrev}.  Then the sequence
$\{\Reg_{\Gamma(\delta_k,g_k)}(g_k)\}$ has a subsequence which converges
strongly to a $J$-minimizing solution of $Ku=g$.
\end{theorem}

\begin{proof}
  By Theorem~\ref{summary:conv}, the sequence
  $\{\Reg_{\Gamma(\delta_k,g_k)}(g_k)\}$ is bounded in $\ell^2$ and
  hence, has a subsequence which converges weakly in
  $\ell^2$. Moreover, it follows from Theorem~\ref{summary:conv} that
  $J(\Reg_{\Gamma(\delta_k,g_k)}(g_k)) \to
  J(u^\dagger)$. By~\cite[Lemma 4.3]{DauDefDeM04} this shows that
  $J(\Reg_{\Gamma(\delta_k,g_k)}(g_k))$ also converges strongly.\qquad
\end{proof}
Note that the entire sequence of iterates converges strongly to the unique $J$-minimizing solution of $Ku=g$ in the case $q\in(1,2]$.

By Theorem~\ref{morozov:convergence}, we also conclude that
$\ell^q$-regularization combined with Morozov's discrepancy principle
gives rise to a (subsequentially) convergent regularization method and,  if additionally the source condition is fulfilled, leads to  convergence rates in the sense of Bregman distances.

Actually, in the latter case, we can strengthen
the above result. More precisely, we can derive convergence rates with respect to the $\ell^q$ norm for $q\in[1,2]$. The two cases $q\in(1,2]$ and $q=1$ have to be treated separately.

In the case $q\in(1,2]$, we take advantage of the differentiability and the high degree of convexity of the functional $J$  to  estimate even the distance between the subgradients appearing in the iterative process.

 The Fenchel conjugate of $J$ is  $J^*(\xi)=\frac{1}{r} \norm{\xi}_{\ell^r}^r$ with $r=q\slash{(q-1)}>2$ (see, e.g., \cite[Proposition 4.2, p. 19]{EkeTem76}.

The following result, which will be useful in the sequel, was pointed out in \cite[Proposition 3.2]{RamRes10}. We give here the proof for the sake of completeness.

\begin{lemma}
If $q\in(1,2]$ and $J$ according to~\eqref{intro:sparse},
then one has for all $v\in\ell^q$ and $u\in\ell^{2(q-1)}$, $u\neq 0$,
\begin{equation}\label{bd_norm}
 D_J(v,u)\geq c_q\norm{v-u}_{\ell^q}^2,
\end{equation}
for $\norm{v-u}_{\ell^q}$ small enough, where $c_q=c_q(u)$ is a positive number.
\end{lemma}
\begin{proof}
The inequality is obvious if $q=2$. Let $q\in(1,2)$. Note that $D(\partial J)=\ell^{2(q-1)}$. In order to simplify the notation in the proof, we omit the subscript for the $\ell^q$ norm. 
Now \cite[Lemma 1.4.8]{ButIus00} implies that for all $v\in\ell^q$, $u\in\ell^{2(q-1)}$
\begin{equation}\label{temp}
 D_J(v,u)\geq(t+\|u\|)^q-\|u\|^q-qt\|u\|^{q-1},
\end{equation}
where $t:=\norm{v-u}$. Let $\varphi(t):=\left(t+\|u\|\right)^q$ for $t$ small enough. The Taylor expansion of $\varphi$ around $0$ yields existence of an $a_t\in(0,t)$ such that
\[
\varphi(t)=\|u\|^q+qt\|u\|^{q-1}+\frac{q(q-1)t^2}{2}\|u\|^{q-2}+\frac{q(q-1)(q-2)t^3}{6}\left(a_t+\|u\|\right)^{q-3}.
\] 
 This inequality and \eqref{temp} imply 
\begin{eqnarray*}
  D_J(v,u)&\geq &\varphi(t)-\|u\|^q-qt\|u\|^{q-1}\\
&=& \frac{q(q-1)t^2}{2}\|u\|^{q-2}+\frac{q(q-1)(q-2)t^3}{6}\left(a_t+\|u\|\right)^{q-3}\\
&=& \frac{q(q-1)t^2}{2}\|u\|^{q-2}\left[1-\frac{(2-q)t}{3}\|u\|^{2-q}\left(a_t+\|u\|^{q-3}\right)\right].
\end{eqnarray*}
Note that $a_t+\|u\|\geq\|u\|$ and $q-3<0$. Hence, $(a_t+\|u\|)^{q-3}\leq\|u\|^{q-3}$ and
\begin{equation}\label{temp1}
 D_J(v,u)\geq  \frac{q(q-1)t^2}{2}\|u\|^{q-2}\left[1-\frac{(2-q)t}{3}\|u\|^{-1}\right].
 \end{equation}
Let $b\in (0,1)$ and take $t<\frac{3(1-b)\|u\|}{2-q}$. Then inequality \eqref{temp1} yields 
\[
D_J(v,u)\geq  c_qt^2,
\]
with $c_q=\frac{bq(q-1)}{2}\|u\|^{q-2}$.\qquad
\end{proof}

\begin{proposition}  Let $K:\ell^2\to H_2$ be linear and bounded,  $J$ be defined
by~\eqref{intro:sparse} with $1< q\leq 2$. Let
  $\Gamma$ be the parameter choice according to Morozov's discrepancy
  principle~\eqref{morozov:disc}. If the $J$-minimizing solution $u^\dagger$ of
  $Ku=g$ satisfies the source condition \eqref{defs:sceqn} with a source element
  $p^\dagger$, then the following convergence rates hold for $k$ sufficiently
  large:
\begin{equation*}\label{rate_q}
\norm{\Reg_{\Gamma(\delta_k,g_k)}(g_k)-u^\dagger}_{\ell^q}=\bigo(\sqrt{\delta_k}),
\end{equation*}
\begin{equation*}\label{rate_r}
\norm{K^* \Reg^*_{\Gamma(\delta_k,
  g_k)}(g_k)-K^* p^\dagger}_{\ell^r}=\bigo(\delta_k^{\frac{q-1}{q}}).
\end{equation*}
\end{proposition}

\begin{proof} We apply inequality \eqref{bd_norm} for
$v=\Reg_{\Gamma(\delta_k,g_k)}(g_k)$ and $u=u^\dagger$ and obtain \[
D_J^{K^*p^\dagger}(\Reg_{\Gamma(\delta_k,g_k)}(g_k),u^\dagger)\geq
c_q\norm{\Reg_{\Gamma(\delta_k,g_k)}(g_k)-u^\dagger}_{\ell^q}^2, \] for $k$ sufficiently
large. This and Theorem \ref{morozov:goodest} imply the first assertion. 

In order to show the estimate for the subgradients, note that (see, e,g., 
\cite[Lemma 1.4.10]{ButIus00}) \[ D_{J^*}(\xi_2,\xi_1)\geq c_r\norm{\xi_2-\xi_1}_{\ell^r}^r \] for any
$\xi_1,\xi_2\in\ell^r$for some positive constant $c_r$ depending on $r\geq 2$.
 Consequently, it follows from Remark \ref{summary:remonrates} that  \[ \norm{K^* \Reg^*_{\Gamma(\delta_k,
  g_k)}(g_k)-K^* p^\dagger}_{\ell^r}=\bigo(\delta_k^{\frac{q-1}{q}})
\] and thus completes the proof.\qquad
\end{proof}

Now we turn to the case of sparse regularization for $q=1$. Here, 
one can derive improved convergence rates in case the solution $u^\dagger$ does
not only fulfill the source condition but also is indeed sparse. To be more
precise, we define for a given set $I\subset \N$ the projection $P_I$
by
\[
(P_I u)_k =
\begin{cases}
u_k,& k\in I\\
0, &  k\notin I.
\end{cases}
\]
and require the following

\begin{ass}\label{sparse:standing_assumption}
\begin{enumerate}[i)]
  \item The solutions $u^\dagger$ of \eqref{intro:constrained} satisfy the source
  condition~\eqref{defs:sceqn} with source element $p^\dagger$.\label{enu:SC}
  \item For $K^*p^\dagger = \xi$ and
  $I = \{k\ |\ \abs{\xi_k}=1\}$, one has that the
  quantity $\theta = \sup\{\abs{\xi_k}\ |\ k\notin I\}$ is strictly smaller than
  one. \label{enu:strict_sparsity_pattern} \item The operator $KP_I:\ell^1\to H$
  is injective in the sense that $P_Iu\neq P_Iv$ implies $KP_I u \neq KP_I
  v$.\label{enu:FBI}
\end{enumerate}
\end{ass}

We start with the following lemma which can
be traced back to~\cite{GraHalSch08} (see also \cite{GraHalSch10}). 

\begin{lemma}\label{lem:estimate_J_below}
Assume that Assumption \ref{sparse:standing_assumption} is satisfied. Then,
there exist constants $\beta_1,\beta_2>0$ such that
\[
J(u)-J(u^\dagger) \geq \beta_1 J(u-u^\dagger) - \beta_2\norm{K(u- u^\dagger)}.
\]
\end{lemma}

\begin{proof}
  Due to Assumption \ref{sparse:standing_assumption}~\ref{enu:FBI}) the
  operator $KP_I$ is injective and hence, there exists $c$ such that
  $\norm{KP_Iu}{}\geq c\norm{P_Iu}_{\ell^1}$ for all $u\in\ell^1$. Now we estimate
  \begin{align*}
J(u- u^\dagger) & = \norm{u-u^\dagger}_{\ell^1} = \norm{P_I(u-u^\dagger)}_{\ell^1} +
\norm{P_{I^\complement} u}_{\ell^1}\\
& \leq \tfrac1c\norm{KP_I(u- u^\dagger)} + \norm{P_{I^\complement} u}_{\ell^1}\\
& \leq \tfrac1c\norm{K(u- u^\dagger)} +
(\norm{K}+1)\norm{P_{I^\complement} u}_{\ell^1}.
\end{align*}
Since $ u^\dagger_k=0$ for $k\notin I$,  Assumption
\ref{sparse:standing_assumption}~\ref{enu:SC})
and~\ref{enu:strict_sparsity_pattern}) implies 
\begin{align*}
\norm{P_{I^\complement}u}_{\ell^1} & = \sum_{k\notin I}\abs{u_k}\\
& \leq \frac{1}{1-\theta}\Bigl(\sum_{k\notin I}\abs{u_k} - \abs{u^\dagger_k} - \xi_k(u_k- u^\dagger_k)\Bigr)\\\
& \leq  \frac{1}{1-\theta}\Bigl(\sum_{k}\abs{u_k} - \abs{ u^\dagger_k} - \xi_k(u_k- u^\dagger_k)\Bigr)\\
& = \frac{1}{1-\theta}\Bigl(J(u) - J(u^\dagger) - \inner{\xi}{u-u^\dagger}\Bigr)\\
& \leq \frac{1}{1-\theta}\Bigl(J(u) - J(u^\dagger)
+\norm{p^\dagger}\norm{K(u-u^\dagger)}\Bigr).
\end{align*}
Combining both estimates gives
\[
J(u-u^\dagger) \leq \bigl(\tfrac1c +
\norm{p^\dagger}\frac{\norm{K}+1}{1-\theta}\bigr)\norm{K(u- u^\dagger)} +
\frac{\norm{K}+1}{1-\theta}(J(u)-J(u^\dagger))
\]
which yields the assertion with
\begin{equation*}
  \beta_1  = \frac{1-\theta}{\norm{K}+1},\qquad
  \beta_2 = \frac{1-\theta}{(\norm{K}+1)c} + \norm{p^\dagger}.
\end{equation*}
\end{proof}
\begin{remark}\upshape
  We remark on Assumption \ref{sparse:standing_assumption}:
  % Assumption~\ref{enu:SC}) is the usual source
  % condition~\eqref{defs:sceqn} which is frequently used to
  % obtain error estimates or convergence rates.
  Statement~\ref{enu:strict_sparsity_pattern}) is related to the
  notion of ``strict sparsity pattern'' in~\cite{BreLor08}. To get a
  practically relevant condition, one may replace this with the
  assumption that the range of $K^*$ is contained in some $\ell^p$
  with $p<\infty$ (since in this case the sequence $\xi$ has to tend
  to zero).  This also implies that $I$ is finite. Alternatively one
  may also work with $K:\ell^2\to H$ (which implies $K^*:H\to
  \ell^2$).

  Assumption~\ref{enu:FBI}) is a restricted injectivity
  condition. Since one needs to know the set $I$ to verify this in
  advance, one often uses the ``finite basis injectivity property''
  (FBI property) from~\cite{BreLor08,Lor08} which states that $KP_I$
  is injective for all finite sets $I$. This condition can be checked
  in advance and hence, it seems more practical.
\end{remark}

% In the following we will always assume that Assumption
% \ref{sparse:standing_assumption} holds. 

Now we treat the case of noisy data and show that the application of
Morozov's discrepancy principle leads to optimal convergence rates.
\begin{theorem}
  Let  $u^\dagger$ be a $J$-minimizing
  solution of $Ku=g$ and assume that
  $\Gamma$ is the parameter choice according to Morozov's discrepancy
  principle~\eqref{morozov:disc}. Then, one has
  \[
  \norm{\Reg_{\Gamma(\delta_k,g_k)}g_k-u^\dagger}_{\ell^1} =
  \bigo(\delta_k).
  \]
\end{theorem}
\begin{proof}
  We estimate the symmetric distance from below using
  Lemma~\ref{lem:estimate_J_below}. To this end, set $u_k =
  \Reg_{\Gamma(\delta_k, g_k)}(g_k)$ and observe that
  \begin{align*}
    D_J^{\text{sym}}(u_k ,u^\dagger) & \geq D_J(u_k ,u^\dagger)\\ & =
    J(u_k) - J(u^\dagger) + \inner{K^*p^\dagger}{u_k-u^\dagger}\\ & \geq \beta_1
    J(u_k-u^\dagger) - \beta_2\norm{Ku_k-g} -
    \inner{p^\dagger}{Ku_k-g}.
  \end{align*}
  Rearranging and using the Cauchy-Schwartz inequality leads to
  \[
  \beta_1 J(u_k-u^\dagger) \leq
  D_J^{\text{sym}}(u_k ,u^\dagger) +
  \left(\beta_2 +
  \norm{p^\dagger}\right)\norm{Ku_k-g}.
  \]
  From the definition of Morozov's discrepancy principle~\eqref{morozov:disc}
  and Theorem~\ref{morozov:goodest} we finally conclude the proof.\qquad
\end{proof}

\subsection{Convergence rate for $n\to\infty$ in the noisefree case}
\label{sec:conv-rate-ntoinfty}

Another consequence of our analysis of the ALM is that we can prove convergence
rates of the ALM iteration with noisefree data which are superior to previous results.
\begin{proposition}\label{sparse:conv_noisefree}
Let  $J$ be according to~\eqref{intro:sparse} with $q=1$, $u^\dagger$ be a $J$-minimizing
solution of $Ku=g$ and $p_0=0$. Then there exists a constant $C>0$
such that the iterates $u_n$ of the ALM fulfill
\[
\norm{u_n-u^\dagger}_{\ell^1} \leq \frac{C}{t_n}.
\]
\end{proposition}
\begin{proof}
  Since $K^*p_n\in\partial J(u_n)$, one has  
  \begin{align*}
J(u_n) - J(u^\dagger) & \leq -\inner{K^*p_n}{u^\dagger - u_n}\\
& = -\inner{p_n}{g - Ku_n}\\
& \leq \norm{p_n}\norm{g-Ku_n}
\end{align*}
Now we use Lemma~\ref{lem:estimate_J_below} to obtain
\begin{align*}
\beta_1 J(u_n-u^\dagger) & \leq J(u_n) - J(u^\dagger) + \beta_2\norm{Ku_n-g}\\
& \leq (\norm{p_n} + \beta_2)\norm{Ku_n-g}.
\end{align*}
Theorem~\ref{summary:ratesest2} (with $\delta = 0$) gives
\[
J(u_n - u^\dagger) \leq \frac{(\gamma + \beta_2)\norm{p^\dagger}}{\beta_1t_n}
\]
which proves the assertion.\qquad
\end{proof}

This proposition shows that the ALM can calculate approximate
solutions to the so-called Basis Pursuit
problem~\cite{chen1998basispursuit} of finding minimal $\ell^1$-norm solutions of underdetermined linear systems
and also gives an estimate on the speed of convergence of the objective
value.

\subsection{Implications for Compressed Sensing}
\label{sec:impl-compr-sens}

Finally we remark on the relation of our results to the theory of compressed sensing:
Linear convergence rates for the variational regularization with
$\ell^1$-norm has been shown in~\cite{GraHalSch08,GraHalSch10} under a source
condition and some assumptions on the operator $K$. A similar result
has been proven (see \cite{CanTao07}) in the finite dimensional setting
of compressed sensing, by using the restricted isometry property
condition. In the latter setting, \cite{GraHalSch10} established
the following connection between the above mentioned conditions - see
\cite[part of Proposition 5.3 and Theorem 4.7]{GraHalSch10} :

\begin{proposition}
Assume that $Ku^\dagger=g$. Assume that $K$ satisfies the $s$-restricted isometry property and let $u^\dagger$ be an $s$-sparse solution of the equation. Then $u^\dagger$ satisfies the source condition and $KP_{I}$ is injective, with  $I$ given by Lemma \ref{lem:estimate_J_below}.
\end{proposition}

Based on this result  and on the ones in  this section, one can immediately state the following:

\begin{proposition}
Assume that $K$ satisfies the $s$-restricted isometry property and let $u^\dagger$ be an $s$-sparse solution of the equation. Then  linear convergence rates  hold for Bregman iterations in the noisy-free case and in the noisy data case when the discrepancy principle is employed.
\end{proposition}

\section{Conclusion}
\label{sec:conclusion}

In this work we showed that Morozov's discrepancy principle \eqref{morozov:disc}
applied to the Augmented Lagrangian Method (ALM) \ref{intro:ala} leads to a
regularization method for linear inverse problems $Ku=g$. This gives a
theoretical justification for the observation that the discrepancy principle
provides useful results in practical situations.

We used a dual characterization of the ALM in order to derive explicit error
bounds for the Bregman distance between the iterates and a true $J$-minimizing solution
$u^\dagger$ of $Ku=g$, if $u^\dagger$ satisfies the source condition
\begin{equation*}
  K^*p^\dagger \in \partial J(u^\dagger)
\end{equation*} 
for a source element $p^\dagger$. In this case, also error bounds for the Bregman distance (with respect to $J^*$)
between the dual iterates in the ALM and $p^\dagger$ were obtained. We also
showed that a \emph{sufficient condition} for the source condition to hold is
the existence of finite accumulation points in the sequence of stopping  indices chosen by the discrepancy principle. 

We applied our general results to particular situations which have a special
appeal for problems arising in imaging. 

Firstly, we considered the case of
\emph{total variation regularization} where we were able to show that the ALM
converges \emph{strictly} in $\BV$ and to establish convergence rates with respect to  an
equivalent metric. 

Secondly, we studied \emph{sparse regularization} on
$\ell^2$, more precisely when $J$ coincides with the $\ell^q$-norm ($q\in[1,2]$). Aside to
$\sqrt{\delta}$-rates in the $\ell^q$-norm for $q>1$, we were able to prove
linear convergence rates for the particular interesting case of $\ell^1$ (under
suitable regularity conditions on $u^\dagger$). The sequence of dual iterates
in the ALM in the latter case carries important information on the support of
the solution. The conjugate function $J^*$ of the $\ell^1$-norm, however,
degenerates to an indicator function. As a consequence, the general estimates
for the dual variables do not reveal much insight in their convergence
behavior. It is still an open issue whether one can obtain more relevant estimates for the dual
variables.

\subsection*{Acknowledgments} The authors thank Martin Burger (University
of M\"unster) for stimulating discussions on the total variation
section. K.F. is supported by the DFG-SNF Research
Group FOR916 \emph{Statistical Regularization} (Z-Project). E.R. acknowledges the
support by Austrian Science Fund, project FWF V82-N118 (Elise Richter
fellowship). D.L. acknowledges support by the DFG under grant LO 1436/2-1 within
the DFG priority program SPP 1324.

%%%%%%%%%%%%%%%%%%%%%%%%%%%%%%%%%%%%%%%%%%%%%%%%%%%%%%%%%%%%%%%%%%%%%%%%%%%%%%%%%%%%%%%%%%%%%%%%%%%%%%%%%%%%%

\bibliographystyle{siam}
\bibliography{./literature}

\end{document}